\documentclass[12pt]{amsart}
\usepackage{amsmath, amscd}
\usepackage{amsthm}
\usepackage{verbatim}
\usepackage{amssymb}
\usepackage[colorlinks=true,pagebackref]{hyperref}
\usepackage[dvipsnames,usenames]{color}
\hypersetup{linkcolor=RawSienna,anchorcolor=BurntOrange,citecolor=OliveGreen,filecolor=BlueViolet,menucolor=Yellow,urlcolor=OliveGreen}

\newtheorem{lemma}{Lemma}[section]
\newtheorem{theorem}[lemma]{Theorem}
\newtheorem{corollary}[lemma]{Corollary}
\newtheorem{proposition}[lemma]{Proposition}

\theoremstyle{definition}
\newtheorem{definition}[lemma]{Definition}
\newtheorem{remark}[lemma]{Remark}

\newtheorem{example}[lemma]{Example}

\newcounter{item-counter}

\theoremstyle{remark}
\newtheorem*{remark*}{Remark}
\newtheorem*{note*}{Note}

\newcommand{\Spec}{\operatorname{Spec}}

\newcommand{\Hom}{\operatorname{Hom}}

\newcommand{\pt}{\operatorname{pt}}

\newcommand{\Lie}{\operatorname{Lie}}
\newcommand{\mc}[1]{\mathcal{#1}}
\newcommand{\s}{{\sigma}}

\newcommand{\CC}{\mathbb{C}}
\newcommand{\ZZ}{\mathbb{Z}}
\newcommand{\RR}{\mathbb{R}}
\newcommand{\QQ}{\mathbb{Q}}

\newcommand{\mbf}[1]{{\mathbf{#1}}} 
\newcommand{\ts}[1]{{\mc{X}(\mathbf{#1})}} 

\begin{document}

\title[Integration on Artin Toric Stacks and Euler characteristics]{Integration on Artin Toric Stacks and Euler characteristics}

\author{Dan Edidin}
\address{Department of Mathematics, University of Missouri-Columbia, Columbia, Missouri 65211}
\email{edidind@missouri.edu}

\author{Yogesh More}
\address{Department of Mathematics, SUNY College at Old Westbury, Old Westbury, NY 11568}
\email{yogeshmore80@gmail.com}
\thanks{The first author was partially supported by NSA grant 
H98230-08-1-0059 while preparing this article.}

\begin{abstract}
  
  There is a well-developed intersection theory on smooth Artin stacks with quasi-affine diagonal
  \cite{Gil:84,Vis:89,EdGr:98, Kre:99}. However, for Artin stacks whose diagonal is
  not quasi-finite the notion of the degree of a Chow cycle is not
  defined. In this paper we propose a definition for the degree of a
  cycle on Artin toric stacks whose underlying toric varieties are
  complete.  As an application we define the Euler characteristic of an Artin toric stack with complete good moduli space - extending the definition of the
  orbifold Euler characteristic. An explicit combinatorial formula is
  given for 3-dimensional Artin toric stacks.
\end{abstract}

\maketitle

\section{Introduction}

Let ${\mathcal X}$ be complete Deligne-Mumford stack. There is a well-developed intersection theory on such stacks \cite{Vis:89, EdGr:98,
  Kre:99} which includes the notion of the degree of 0-cycle. Over an
algebraically closed field of characteristic 0 if $x$ is a point of
${\mathcal X}$ with stabilizer $G_x$ then the degree of the 0-cycle
$[x] \in A_0({\mathcal X})$ is $1/|G_x|$. As a consequence one can
compute the integral of an algebraic cohomology class on a complete
Deligne-Mumford stack. This theory of integration plays a crucial
role in Gromov-Witten, Donaldson-Thomas and other invariants in modern
algebraic geometry.

A natural problem is to extend the notion of degree to stacks with
non-finite diagonal. While such stacks are not separated (and hence
cannot be proper over the ground field) there is a class of Artin
stacks which in many ways behave like complete Deligne-Mumford
stacks. Specifically we consider Artin stacks which are generically
Deligne-Mumford and have a complete good moduli space (in the sense of
\cite{Alp:08}). If ${\mathcal X}$ is such a stack then there is a
well-defined pushforward of Grothendieck groups $p_*\colon K({\mathcal
  X}) \to K(pt) = \ZZ$ thereby defining  Euler characteristics of
sheaves on such stacks. Thus it is natural to expect there to be a corresponding
pushforward for algebraic cycles as well as a Riemann-Roch theorem
relating the $K$-theoretic and the Chow-theoretic pushforward via some
analogue of the Chern character.

The purpose of the present paper is to take a first step in this
direction by defining the degree of a 0-cycle on a toric Artin stack
in the sense of \cite{BCS:05}.  Our definition is based on our
previous paper \cite{EdMo:10} where we showed that if ${\mathcal
  X}({\mathbf \Sigma})$ is a toric Artin stack with complete good moduli space
$X(\Sigma)$ then there is a canonical birational morphism of toric stacks
${\mathcal X}({\mathbf \Sigma'}) \stackrel{f} \to {\mathcal X}({\mathbf \Sigma})$ such
that ${\mathcal X}({\mathbf \Sigma'})$ is Deligne-Mumford and such that the
induced morphism of the underlying toric varieties is proper and birational. The degree of 
a 0-cycle on ${\mathcal X}({\mathbf \Sigma})$ is then defined as the degree of its
pullback to the Deligne-Mumford stack ${\mathcal X}(\Sigma')$.

As an application we define {\em the stacky Euler characteristic} 
of an Artin toric stack
to be the degree of the Euler class of the tangent bundle stack,
and we give a
combinatorial formula for this Euler characteristic in dimension 3.

Along the way we prove that the integral Chow ring of an Artin toric
stack equals the Stanley--Reisner ring of the associated fan. Iwanari \cite{Iwa:09} proves this for Deligne-Mumford toric stacks; our proof is different.

\section{Toric stacks and Stanley--Reisner rings}

In this section, we recall the definition of toric stacks and in the process establish the notation we will use for toric stacks. As in \cite{Iwa:09}, for simplicity we will restrict our attention to stacky fans for which $N$ is torsion free. A {\em stacky fan}
${\mathbf\Sigma}=(N, \Sigma, \{v_1, \ldots, v_{n}\})$ consists of:
\begin{list}{\arabic{item-counter}.}{\usecounter{item-counter}}
\item  a finitely generated free abelian group $N=\ZZ^d$,
\item a (not neccesarily simplicial) fan $\Sigma \subset N_\QQ$, whose rays we will denote by $\rho_1, \ldots, \rho_n$,
\item for every $1 \leq i \leq n$, an element $v_i \in N$ such that $v_i$ lies on the ray $\rho_i$.
\end{list}

\subsection{Toric stacks via the Cox construction}\label{subsec:Cox-constr}
We now recall the construction \cite{BCS:05} of the toric stack
$\mc{X}(\mathbf{\Sigma})$ associated to a stacky fan
${\mathbf\Sigma}=(N, \Sigma, \{v_1, \ldots, v_{n}\})$. Let
$S=\CC[X_\rho \mid \rho \in \Sigma(1)]$, where $\Sigma(1)$ denotes the
rays in $\Sigma$. Define the ideal (of S) $B(\Sigma)=(
X_{\widehat{\s}} \mid \sigma \in \Sigma)S$, where
$X_{\widehat{\s}}=\prod_{\rho \in \Sigma(1) \smallsetminus \s}
X_{\rho}$.  The Cox space $C(\Sigma)$ is defined to be $\Spec S
\smallsetminus {V}(B(\Sigma))$, where ${V}(B(\Sigma))$ is the
vanishing locus of the ideal $B(\Sigma)$. Let $$G({\mathbf \Sigma})=\{
(t_\rho)_{\rho \in \Sigma(1)} \in (\CC^*)^n \mid \prod_{\rho \in
  \Sigma(1)} t_\rho^{\langle m, v_\rho \rangle}=1 \ \ \forall m\in M
\},$$ where $M=\Hom(N, \ZZ)$ and $\langle \cdot, \cdot \rangle: M
\times N \to \ZZ$ is the natural pairing. The group $G({\mathbf
  \Sigma})$ acts on $C(\Sigma)$ by restricting the natural action of
$(\CC^*)^n$ on $C(\Sigma) \subseteq \CC^n=\Spec S$. The {\em toric stack}
$\mc{X}(\mathbf{\Sigma})$ associated to the stacky fan
${\mathbf\Sigma}=(N, \Sigma, \{v_1, \ldots, v_{n}\})$ is the quotient
stack $[C(\Sigma)/G(\mathbf{\Sigma})]$.

\subsection{Chow rings of Artin toric stacks}

Our first result is a description of the Chow ring of an Artin toric stack in terms of the combinatorics of its stacky fan. 

\begin{definition}\cite[Definition
  2.1]{Iwa:09}\label{def:stanley-reisner-ring}
  Given a stacky fan $\mathbf{\Sigma}=(N, \Sigma, \{v_\rho\}_{\rho \in
    \Sigma(1)}\})$, let $S(\Sigma)=\ZZ[x_\rho \mid \rho \in
  \Sigma(1)]$.  Let $I_{\mbf{\Sigma}}$ be the ideal of $S(\Sigma)$
  generated by $\sum_{\rho \in \Sigma(1)} \langle m, v_\rho \rangle
  x_\rho$ as $m$ ranges over $M$. Let $J_\Sigma$ be the ideal of
  $S(\Sigma)$ generated by monomials $x_{\rho_1}\cdots x_{\rho_s}$
  such that the rays $\rho_1, \ldots, \rho_s$ are not contained in any
  cone in $\Sigma$:
\begin{equation}\label{eqn:stanley-reisner-ideal}
J_\Sigma = \langle x_{\rho_1}\cdots x_{\rho_s} \mid \{ \rho_1, \ldots, \rho_s \} \nsubseteq \tau \in \Sigma \rangle
\end{equation}
We will call $J_\Sigma$ the Stanley--Reisner ideal of
$\Sigma$. Following \cite[Definition 2.1]{Iwa:09} (but not exactly the
possibly more standard terminology in \cite[Definition
12.4.10]{CLS:11}, see also \cite{MS:05}), define the Stanley--Reisner
ring $SR(\mbf{\Sigma})$ to be
\begin{equation}\label{eqn:stanley-reisner-ring}
SR(\mbf{\Sigma})=S(\Sigma)/(I_{\mbf{\Sigma}} + J_\Sigma)
\end{equation}
\end{definition}

Suppose a group $G$ acts on a smooth scheme $X$, and let $\mc{X}=[X/G]$ be the resulting quotient stack. The integral Chow ring $A^*(\mc{X})$ of $\mc{X}$ is the $G$-equivariant Chow ring $A^*_G(X)$ of $X$ \cite{EdGr:98}.  
We prove that the integral Chow ring of an Artin toric stack
$\ts{\Sigma}$ is the Stanley--Reisner ring of the stacky fan
$\mbf{\Sigma}$. Iwanari \cite{Iwa:09} proved the same result for toric
Deligne-Mumford stacks. Our proof is different: rather than computing
the equivariant Chow groups directly from the definition, we use the
excision sequence for equivariant Chow groups.

\begin{proposition}\label{prop:chow-ring}
  Let $\mathbf{\Sigma}=(N, \Sigma, \{v_\rho\}_{\rho \in \Sigma(1)})$
  be a stacky fan. There is an isomorphism between the integral Chow
  ring $A^*(\ts{\Sigma})$ of the smooth toric Artin stack
  $\mc{X}(\mathbf{\Sigma})$ and the Stanley--Reisner ring
  $SR(\mbf{\Sigma})$, sending $[V(X_\rho)]_{G({\mathbf \Sigma})} \to
  x_\rho$. Here $[V(X_\rho)]_{G({\mathbf \Sigma})} \in
  A^*_{G({\mathbf \Sigma})}(C(\Sigma))=A^*(\ts{\Sigma})$ denotes the equivariant
  fundamental class of the divisor $V(X_\rho)$ of $C(\Sigma)$.
\end{proposition}

\begin{proof}
As before, let $S(\Sigma)=\ZZ[x_\rho \mid \rho \in \Sigma(1)]$, $S=\CC[X_\rho \mid \rho \in \Sigma(1)]$, and $G=G({\mathbf \Sigma})$.

The excision sequence in equivariant Chow groups is the exact sequence:

$$A^*_G({V}(B(\Sigma))) \to A^*_G(\Spec S)  \to A^*_G(C(\Sigma)) \to 0$$ 

Since $\Spec S$ is an affine bundle over $\pt=\Spec \CC$, we have
$A^*_G(\Spec S)=A^*_G(\pt)=S(\Sigma)/I_{\mbf{\Sigma}}$. By the
following lemma, the image of $A^*_G({V}(B(\Sigma)))$ in $A^*_G(\Spec
S)$ is the Stanley--Reisner ideal $J_\Sigma$ (equation
\ref{eqn:stanley-reisner-ideal}). Hence $A^*_G(C(\Sigma))=
S(\Sigma)/(I_{\mbf{\Sigma}} + J_\Sigma)=SR(\mbf{\Sigma})$, the
Stanley--Reisner ring of $\Sigma$.
\end{proof}

\begin{lemma}
Let $B(\Sigma)=(X_{\widehat{\s}} \mid \sigma \in \Sigma)S$ be as in Section \ref{subsec:Cox-constr}. Then we have a primary decomposition 
\begin{equation}\label{primary-decomp}
B(\Sigma)=\bigcap_{\{ \rho_1, \ldots, \rho_s \} \nsubseteq \tau \in \Sigma} (X_{\rho_1}, \ldots, X_{\rho_s})
\end{equation}
where the intersection ranges over sets of rays $\{\rho_1, \ldots, \rho_s\}$ that are not contained in some cone in $\Sigma$. 
\end{lemma}

\begin{proof}
  First we show the ``$\subseteq$'' inclusion. Fix $\s \in
  \Sigma$. For any $\{ \rho_1, \ldots, \rho_s \}$ not contained in a
  cone in $\Sigma$, there is some $\rho_i$ (where $1 \leq i \leq s$) such that $\rho_i \notin
  \sigma$, (for otherwise $\{ \rho_1, \ldots, \rho_s \} \subseteq
  \s$). Hence $X_{\widehat{\s}} \in (X_{\rho_1}, \ldots, X_{\rho_s})$,
  and hence we have proved the ``$\subseteq$'' inclusion.

  To show the opposite inclusion, let $RI$ denote the ideal on the
  right hand side of equation \ref{primary-decomp}. Since $RI$ is a
  monomial ideal, it suffices to show that monomials in $RI$ belong to
  $B(\Sigma)$, i.e. we need to show if a monomial $X_{\rho_1}\cdots
  X_{\rho_s}$ lies in $RI$, then it lies in $B(\Sigma)$. We will show
  the contrapositive, namely that if $X_{\rho_1}\cdots X_{\rho_s}
  \notin B(\Sigma)$, then $X_{\rho_1}\cdots X_{\rho_s} \notin RI$. If
  $X_{\rho_1}\cdots X_{\rho_s} \notin B(\Sigma)$, then the set
  $\Sigma(1) \smallsetminus \{\rho_1, \ldots, \rho_s \}$ is not
  contained in any cone of $\Sigma$ (because if $\Sigma(1)
  \smallsetminus \{\rho_1, \ldots, \rho_s \} \subseteq \tau(1)$ for
  some cone $\tau \in \Sigma$, then taking complements gives
  $\{\rho_1, \ldots, \rho_s \} \supseteq \Sigma(1) \smallsetminus
  \tau(1)$, and hence $X_{\rho_1}\cdots X_{\rho_s}$ would be a
  multiple (in $S$) of $X_{\widehat{\tau}} \in B(\Sigma)$). If we let
  $\{q_1, \ldots, q_t\}= \Sigma(1) \smallsetminus \{\rho_1, \ldots,
  \rho_s \}$ then $X_{\rho_1}\cdots X_{\rho_s} \notin (X_{q_1},
  \ldots, X_{q_t})$, and hence $X_{\rho_1}\cdots X_{\rho_s} \notin
  RI$.

\end{proof}

In the remainder of this article, by virtue of the isomorphism
$SR(\mbf{\Sigma}) \simeq A^*(\ts{\Sigma})$ given in
Prop. \ref{prop:chow-ring}, we will write $x_\rho$ to mean
$[V(X_\rho)]_{G({\mathbf \Sigma})} \in A^*(\ts{\Sigma})$ (where $\rho \in
\Sigma(1)$).

\section{Integration on complete toric Deligne-Mumford stacks}
In this section, we will study intersection theory on complete toric 
Deligne-Mumford stacks.
In this case,
we have an isomorphism $\phi:A^*(\ts{\Sigma})_\QQ \to
A^*(X(\Sigma))_\QQ$ \cite{EdGr:98} and  for the rest of this paper, we
will work with the rational (as opposed to integral) Chow ring
$A^*(\ts{\Sigma})_\QQ$ of toric stacks.

\begin{definition}[$D_{\s, \mbf{\Sigma}}$, stacky multiplicity of a cone]\label{def:stacky-multiplicity}
  Let $N$ be a lattice and let $\s$ be simplicial cone $N_\RR$, and
  let $s=\dim \s$. Suppose $v_1, \ldots, v_s$ are elements of $N$ such
  that $\s=\RR_{\geq 0}\langle v_1, \ldots, v_s\rangle$ (in
  particular, each $v_i$ is a lattice point on exactly one ray of
  $\s$). Let $N_\s= \ZZ \langle N \cap \s \rangle$ be the lattice
  generated by $N \cap \s$. We define the (stacky) multiplicity of the
  stacky cone $\overline{\sigma}=(N, \s, \{v_1, \ldots, v_s\})$ to be
  the order of $N_\s/\ZZ\langle v_1, \ldots, v_s\rangle$, and denote
  this positive integer by $D_{\overline{\s}}$.  Let
  $\mathbf{\Sigma}=(N, \Sigma, \{v_1, v_2, \ldots, v_n\})$ be a stacky
  fan and $\s$ a simplicial cone in $\Sigma$. Label the $v_i$ so that
  $v_1, \ldots, v_s$ are on the rays of $\s$. Let $\overline{\s}=(N,
  \s, \{v_1, \ldots, v_s\})$. Define the stacky multiplicity of $\s$
  with respect to $\mbf{\Sigma}$ to be $D_{\overline{\s}}$, and denote
  this quantity by $D_{\s, \mbf{\Sigma}}$.
\end{definition}

\begin{proposition}\cite[p.3749]{EdGr:03}\label{prop:dan-isom-thm-toric}
Let $\mathbf{\Sigma}$ be a complete simplicial stacky fan, and let $\pi:\ts{\Sigma} \to X(\Sigma)$ be the coarse moduli space of $\ts{\Sigma}$. Given distinct rays $\rho_1, \ldots, \rho_s$ of $\Sigma$ forming a cone $\s=\langle \rho_1, \ldots, \rho_s \rangle$ in $\Sigma$, let $V(\s) \subseteq X(\Sigma)$ be the orbit closure associated to $\s$.  Let $\phi:A^*(\ts{\Sigma})_\QQ \to A^*(X(\Sigma))_\QQ$ be the isomorphism of Chow rings (\cite[Theorems 3,4]{EdGr:98}). Then  
\begin{equation}\label{eqn:dan-isomorphism-formula}
\phi(x_{\rho_1} \cdots x_{\rho_s})=\frac{1}{D_{\s, \mbf{\Sigma}}}[V(\s)],
\end{equation} 
where $D_{\s, \mbf{\Sigma}}$ is the stacky multiplicity of $\s$ in $\mbf{\Sigma}$ (Definition \ref{def:stacky-multiplicity}).
\end{proposition}

\subsection{Self-intersections on toric Deligne-Mumford stacks}
If $\rho_1, \ldots, \rho_s$ are not distinct, then to compute
$\phi(x_{\rho_1} \cdots x_{\rho_s})$, we can express $x_{\rho_1}
\cdots x_{\rho_s}$ as a rational linear combination of monomials with
no factor $x_\rho$ appearing more than once, and then we can apply
Proposition \ref{prop:dan-isom-thm-toric} to compute $\phi$. The
following theorem gives a procedure for
finding such a linear combination.

\begin{theorem}[Recursive formula for
  self-intersections]\label{prop:general-equiv-self-intersection}
  Let $\mathbf{\Sigma}=(\ZZ^d, \Sigma, \{v_\rho \}_{\rho \in
    \Sigma(1)})$ be a complete simplicial stacky fan.  Let $\rho_1,
  \ldots, \rho_s$ be distinct rays in $\Sigma$, and let $a_1, \ldots,
  a_s$ be positive integers and assume $a_{i_0} \geq 2$ for some $1
  \leq i_0 \leq s$. If $\rho_1, \ldots, \rho_s$ are not contained in a
  cone in $\Sigma$, then $\prod_{i=1}^s x_{\rho_i} =0$ (in
  $A^*(\ts{\Sigma}$) and hence $\prod_{i=1}^s x_{\rho_i}^{a_i}
  =0$. Otherwise, fix a (maximal) $d$-dimensional cone $\s=\langle
  \rho_1, \ldots, \rho_s, \rho_{s+1}, \ldots, \rho_d\rangle \in
  \Sigma(d)$ containing $\rho_1, \ldots, \rho_s$. Let $C = \{ \rho \in
  \Sigma(1) \smallsetminus \{\rho_1, \ldots, \rho_d \} \mid \langle
  \rho, \rho_1, \ldots, \rho_s \rangle \in \Sigma\}$ (note $C \neq
  \emptyset$ since $\Sigma$ is complete).

For a ray $\rho \in C$, let $b_{\rho_i, \rho}$ be rational numbers such that
\begin{equation}\label{eqn:gen-relation}
v_{\rho} +  \sum_{i=1}^d b_{\rho_i, \rho} v_{\rho_i} =0
\end{equation}
(Such a relation always exists and is unique since $v_{\rho_1}, \ldots, v_{\rho_d}$ form a basis of $\RR^d$.) 

Then 
\begin{equation}\label{eqn:recursive-formula}
\prod_{i=1}^{s} x_{\rho_i}^{a_i}= \sum_{\rho \in C}{b_{\rho_{i_0}, \rho}}x_{\rho}x_{\rho_{i_0}}^{a_{i_0}-1}\prod_{i=1, i \neq i_0 }^sx_{\rho_i}^{a_i}
\end{equation}
in $A^*(\ts{\Sigma})_\QQ.$

\end{theorem}

\begin{remark}
  Note that on the right-hand side of eq. \ref{eqn:recursive-formula},
  the exponent of $x_{\rho_{i_0}}$ is one less than on the left-hand
  side, but for $i \neq i_0$ (and $1\leq i \leq s$) the exponent of
  $x_{\rho_i}$ is the same on both sides. Since there was nothing
  special about $x_{\rho_1}$ (other than $a_1 \geq 2$), we can
  repeatedly apply Theorem \ref{prop:general-equiv-self-intersection} to
  express any monomial $\prod_{i=1}^{s} x_{\rho_i}^{a_i}$ as a
  rational linear combination of monomials with no factor $x_\rho$
  appearing more than once. We can then apply Proposition
  \ref{prop:dan-isom-thm-toric} to compute $\phi$. Hence Theorem
  \ref{prop:general-equiv-self-intersection} gives a method to compute
  the integral of (or more generally, $\phi$ applied to) any Chow
  cohomology class on $\ts{\Sigma}$. We give an example after the
  proof.
\end{remark}

\begin{proof}
  For notational convenience, assume $i_0=1$. Let
  $P=x_{\rho_{1}}^{a_{1}-1}\prod_{i=2 }^sx_{\rho_i}^{a_i}$.  For every
  $m \in M$, we have $\sum_{\rho \in \Sigma(1)} \langle m, v_\rho
  \rangle x_\rho =0$ in the Chow ring $A^*(\ts{\Sigma})_\QQ$.
  Multiplying by $P$ we get
\begin{equation}\label{eqn:basic-linear-reln}
 \sum_{\rho \in C} \langle m, v_{\rho} \rangle x_{\rho} P + \sum_{i=1}^d \langle m, v_{\rho_i} \rangle x_{\rho_i} P=0
\end{equation}
where we used $x_\rho P =0$ for $\rho \notin C \cup \{\rho_1, \ldots,
\rho_d \}$ (by Proposition \ref{prop:chow-ring}). By eq. \ref{eqn:gen-relation}

$$v_{\rho} =  - \sum_{i=1}^d {b_{\rho_i, \rho}} v_{\rho_i} $$
Substituting this expression for $v_{\rho}$ in the first term of
eq. \ref{eqn:basic-linear-reln} gives

\begin{equation}
  \sum_{\rho \in C} \langle m, -\sum_{i=1}^d {b_{\rho_i, \rho}} v_{\rho_i} \rangle x_{\rho} P + +\sum_{i=1}^d \langle m, v_{\rho_i} \rangle x_{\rho_i} P =0
\end{equation}
and collecting the coefficients of $\langle m, v_{\rho_i} \rangle$ gives:

\begin{equation}\label{eqn:almost}
\sum_{i=1}^d \langle m, v_{\rho_i} \rangle (x_{\rho_i} P - \sum_{\rho \in C}{b_{\rho_i, \rho}}x_{\rho}P) =0
\end{equation}

Equation \ref{eqn:almost} holds for all $m \in M$. Since  $v_{\rho_1}, \ldots, v_{\rho_d}$ are linearly independent, we conclude that $x_{\rho_i} P - \sum_{\rho \in C}{b_{\rho_i, \rho}}x_{\rho}P =0 \in A^*(\ts{\Sigma})_\QQ$ for $1 \leq i\leq d$.

In particular, taking $i=1$, we conclude $x_{\rho_1} P = \sum_{\rho \in C}{b_{\rho_1, \rho}}x_{\rho}P$, i.e.
$$\prod_{i=1}^{s} x_{\rho_i}^{a_i}= \sum_{\rho \in C}{b_{\rho_1, \rho}}x_{\rho}x_{\rho_{1}}^{a_{1}-1}\prod_{i=2 }^sx_{\rho_i}^{a_i}$$

\end{proof}

\begin{example}\label{example:triple-intersection}
Let $v_1=(1,0,1)$, $v_2=(0, 2, 1)$, $v_3=(-1,0, 1)$, $v_4=(0,-1,1)$, $v_5=(0,0,-1)$. Let $\Sigma \subseteq \RR^3$ be the fan whose maximal cones are (the $\RR_{\geq 0}$-spans of) $\s=\langle v_1, v_2, v_3, v_4 \rangle$, and for $(i,j) \in \{ (1, 2), (2, 3), (3, 4), (1, 4) \}$, $\s_{ij5}=\langle v_i, v_j, v_5 \rangle$. The only nonsimplicial cone is $\s$ and its star subdivision is formed by splitting it into four  $3$-cones, each having the ray through $v_0=v_1+v_2+v_3+v_4=(0,1,4)$ as a ray. The cones in $\Sigma_\s(3) \smallsetminus \Sigma$ are $\s_{0ij}=\langle v_0, v_i, v_j \rangle$ for $(i,j) \in \{ (1, 2), (2, 3), (3, 4), (1, 4) \}$. Let $\mbf{\Sigma}=(\ZZ^3, \Sigma, \{v_1, v_2, v_3, v_4, v_5 \})$ and $\mbf{\Sigma_\s}=(\ZZ^3, \Sigma_\s, \{v_0, v_1, v_2, v_3, v_4, v_5 \})$ be the corresponding toric stack. For brevity, let $Dijk=D_{\s_{ijk}, \mbf{\Sigma_\s}}$.
Computing the determinants of the appropriate triples of lattice vectors $v_i$, we get that $D125=D235=2$, $D345=D145=1$, $D012=D023=7$, $D034=D014=5$.

For brevity let $x_i=x_{\langle v_i \rangle} \in
A^*(\ts{\Sigma_s})_\QQ$. We use Theorem
\ref{prop:general-equiv-self-intersection} to compute $\phi(x_0^2)$
and $\phi(x_0^3)$.  We will extend $v_0$ to the basis $v_0, v_1,
v_2$. Then $C=\{\langle v_3\rangle, \langle v_4\rangle\}$.  A
calculation shows that eq. \ref{eqn:gen-relation} reads $v_3+
\frac{-4}{7}v_0 + v_1 + \frac{2}{7}v_2=0$ and
$v_4+\frac{-3}{7}v_0+0\cdot v_1 + \frac{5}{7}v_2=0$. Hence
eq. \ref{eqn:recursive-formula} reads $x_0^3=\frac{-4}{7}x_0^2x_3 +
\frac{-3}{7}x_0^2x_4$. We have to do another iteration to simplify
$x_0^2x_3$ and $x_0^2x_4$. This time, we extend the basis $v_0, v_3$
to $v_0, v_3, v_4$. Then eq. \ref{eqn:gen-relation} becomes $v_1 +
\frac{-2}{5}v_0 + v_3 + \frac{-2}{5}v_4 =0$, and $v_2 +
\frac{-3}{5}v_0 + 0\cdot v_3 + \frac{7}{5}v_4 =0$. Then
eq. \ref{eqn:recursive-formula} gives $x_0^2x_3=\frac{-3}{5}x_0x_2x_3$
and $x_0^2x_4=\frac{-2}{5}x_0x_1x_4$. Hence $x_0^3 =
\frac{12}{35}x_0x_2x_3 + \frac{6}{35}x_0x_2x_4$, and we have
eliminated all the self intersections. Then by Proposition
\ref{prop:dan-isom-thm-toric} $\int_{\ts{\Sigma_\s}} x_0^3 =
\phi(x_0^3)= \frac{12}{35}\frac{1}{7} +
\frac{6}{35}\frac{1}{5}=\frac{102}{1225}.$

Similarly, $x_0^2=\frac{-4}{7}x_0x_3 + \frac{-3}{7}x_0x_4$. Applying Proposition \ref{prop:dan-isom-thm-toric} gives the relation $[V(\rho_0)]^2 = \frac{-4}{7}[V(\langle v_0, v_3 \rangle)] + \frac{-3}{7}\frac{1}{5}[V(\langle v_0, v_4 \rangle)]$ in $A^*(X(\Sigma_\s))_\QQ$ in the rational Chow group of the toric variety $X(\Sigma_\s)$.
\end{example}

As a special case of Theorem \ref{prop:general-equiv-self-intersection}, we recover a formula for $x_{\rho_1}^2x_{\rho_2} \cdots x_{\rho_{d-1}}$, equivalent to the one given in \cite[Prop. 6.4.4]{CLS:11}: 

\begin{corollary}\label{cor:cox-result}
Let $\mathbf{\Sigma}=(\ZZ^d, \Sigma, \{v_\rho \}_{\rho \in \Sigma(1)})$ be a complete simplicial stacky fan. Let $\tau^+$ and $\tau^-$ be distinct (maximal) $d$-dimensional cones of $\Sigma$ such that $\s=\tau^+ \cap \tau^-$ is a cone of dimension $d-1$. Let $\rho^+$ and $\rho^-$ be the two rays of $\Sigma$ such that $\tau^+=\langle \rho^+, \s \rangle$ and $\tau^-=\langle \rho^-, \s \rangle$. Let $\beta^+, \beta^-, b_\rho$ be rational numbers such that we have a relation 
\begin{equation}\label{eqn:relation}
\beta^+v_{\rho^+} + \beta^-v_{\rho^-} + \sum_{\rho \in \s(1)} b_\rho v_\rho =0
\end{equation}

Then for any $\rho \in \s(1)$, $$x_\rho x_\s = \frac{b_\rho}{\beta^+}x_{\rho^+}x_\s$$ in $A^*(\ts{\Sigma}).$ 
\end{corollary} 

\begin{proof}
  Simply take $s=d-1$, $a_1=2$, $a_i=1$ for $2 \leq i \leq s$ in
  Theorem \ref{prop:general-equiv-self-intersection}.
\end{proof}

\begin{remark}
Applying the ring isomorphism $\phi:A^*(\ts{\Sigma})_\QQ \to A^*(X(\Sigma))_\QQ$ to the result in Corollary \ref{cor:cox-result} gives $ \phi(x_\rho)\phi( x_\s) = \frac{b_\rho}{\beta^+}\phi(x_{\rho^+}x_\s)$. By Prop. \ref{prop:dan-isom-thm-toric}, this equality becomes $$[V(\rho)]\cdot[\frac{1}{D_{\s, \mbf{\Sigma}}}V(\s)] = \frac{b_\rho}{\beta^+}[\frac{1}{D_{\tau^+, \mbf{\Sigma}}}V(\tau^+)],$$ which rearranges to the formula given in \cite[Prop. 6.4.4]{CLS:11}:
$$[V(\rho)]\cdot[V(\s)] = \frac{b_\rho}{\beta^+}\frac{D_{\s, \mbf{\Sigma}}}{D_{\tau^+, \mbf{\Sigma}}}[V(\tau^+)],$$
for any $\rho \in \s(1)$.
\end{remark}


\subsection{Euler characteristic of toric Deligne-Mumford stacks}

\begin{definition} \label{def.eulerchar}
If  ${\mathcal X}$ is an $n$-dimensional smooth complete Deligne-Mumford stack
 we define the Euler characteristic of ${\mathcal X}$ as $
\chi({\mathcal X}) := \int_{\mathcal X} c_{n}({\mathbb T})$
where ${\mathbb T}$ is the tangent bundle of ${\mathcal X}$.
\end{definition}
\begin{remark}
  The Euler characteristic we define is a rational number and should
  not be confused with the orbifold Euler characteristic defined by
  physicists.  There is an orbifold version of the Gauss-Bonet
  theorem which implies that our definition is equivalent to the
  following topological definition for stacks of the form $[M/G]$, where
$G$ is a finite group and M is a manifold. The manifold $M$ has a CW decomposition where the 
stabilizer has constant order $g_c$ on the interior of each cell $c$. Then $\chi([M/G]) := \sum_{c} { (-1)^{\dim c}\over{g_c}}$.
Since every Deligne-Mumford stack (with finite stabilizer) 
is locally of the form $[M/G]$ the definition of Euler characteristic can be extended to arbitrary Deligne-Mumford stacks, but we do not need this here.
\end{remark}

The following is well-known but we include a proof for lack of a suitable reference.
\begin{proposition}\label{prop:simplicial-euler-char}
Let $\mbf{\Sigma}=(N, \Sigma, \{v_1, v_2, \ldots, v_n\})$ be a complete simplicial stacky fan (i.e $\Sigma$ is simplicial and complete). Let $\Sigma_{\max}$ denote the set of maximal cones in $\Sigma$. For a maximal cone $\s \in \Sigma_{\max}$, let $D_\s$ denote the multiplicity of $\s$. Then $$\chi(\mc{X}(\mbf{\Sigma})) = \sum_{\s \in \Sigma_{\max}} \frac{1}{D_{{\s}, \mbf{\Sigma}}}$$
\end{proposition}

\begin{proof}
Let $\mc{X}=\mc{X}(\mathbf{\Sigma)}=[C/G]$ where $C=C(\Sigma) \subseteq \mathbb{A}^n=\Spec \CC[X_1, \ldots, X_n]$ is the Cox space of $\Sigma$, and $G=G(\mbf{\Sigma})$ is as in Section \ref{subsec:Cox-constr}. The coarse moduli space of $\mc{X}$ is $X=X(\Sigma)$, the toric variety associated to $\Sigma$. 
The equivariant tangent bundle $T_C$ is $\oplus \mc{O}(X_i)$ (\cite[p.3751]{EdGr:03})
so $$\chi(\mc{X}(\mbf{\Sigma})) = \int_{C} \prod_{i=1}^n (1+x_i) = \int_X \phi(\prod (1+x_i))$$

If $v_{i_1}, \ldots, v_{i_s}$ are not contained in a cone of $\Sigma$, then $[V(X_{i_1}, \ldots, X_{i_s})]_G =0$ in $A^*_G(C)$ (Proposition \ref{prop:chow-ring}). Hence if we expand the product $\prod_{i=1}^n (1+x_i)$, the only monomials $x_{i_1} \cdots x_{i_s}$ that can contribute to $\int_X \phi(\prod 1+x_i)$ are ones where $v_{i_1}, \ldots, v_{i_s}$ span a maximal cone (we require maximal cones since the integral only considers the dimension $0$ part of the cycle). In this case, the image of $V(X_{i_1}, \ldots, X_{i_s})$ is a point in $X(\Sigma)$, and Prop. \ref{prop:dan-isom-thm-toric} tells us to assign a factor of $\frac{1}{D_{{\s}, \mbf{\Sigma}}}$ to the class $[\pt]$ of this point. The usual identification of  $A_0(X(\Sigma))$ with $\ZZ$ sending $[\pt] \in A^*(X(\Sigma))$ to $1 \in \ZZ$ gives the result.
\end{proof}

\begin{remark}\label{rem:simplicial-euler-char} 
Note that Prop. \ref{prop:simplicial-euler-char} recovers the result (e.g. \cite[Theorem 12.3.9]{CLS:11}) that the Euler characteristic of a smooth toric variety is the number of maximal cones. 
\end{remark}

The following result gives expresses how the Euler characteristic increases after a stacky star subdivision \cite[Definition 4.1]{EdMo:10}
of a simplicial stacky fan.

\begin{lemma}
  Let $\mathbf{\Sigma}=(\ZZ^d, \Sigma, \{v_1, \ldots, v_n\})$ be a
  complete simplicial stacky fan and let $\s$ be a cone in
  $\Sigma$. Let $\mathbf{\Sigma_\s}=(\ZZ^d, \Sigma_\s, \{v_0, v_1,
  \ldots, v_n\})$ be the stacky fan formed by stacky star subdivision
  of $\mathbf{\Sigma}$ with respect to $\s$. Then $\chi(\ts{\Sigma_\s})=\chi(\ts{\Sigma}) + \frac{s-1}{D_{\s, \mbf{\Sigma}}}$, where
  $s=\dim \s$.
\end{lemma}

\begin{proof}
Label the $v_i$ such that $\s=\langle v_1, \ldots, v_s\rangle$. Hence $v_0=v_1 + \cdots + v_s$. For $1\leq i \leq s$, let $\s_i =\langle v_1, \ldots, \widehat{v_i}, \ldots, v_s\rangle$ be the cone of $\Sigma_\s$ spanned by $v_0$ and all the rays of $\s$ except for $\langle v_i \rangle$. Then by properties of determinants, we have $D_{\s_i, \mbf{\Sigma_\s}}=D_{\s, \mbf{\Sigma}}$. Then apply Prop. \ref{prop:simplicial-euler-char}.
\end{proof}

\section{Integration on Artin toric stacks}\label{subsec:Artin-integrals}
We now come to the main definitions of the paper.  Let ${\mathbf
  \Sigma}$ be a (not necessarily simplicial) stacky fan and let
${\mathcal X}({\mathbf \Sigma})$ be the associated Artin toric stack. By
\cite[Theorem 5.2]{EdMo:10} there is a simplicial stacky fan
$\mbf{\Sigma'}$ canonically obtained from ${\mathbf \Sigma}$ by stacky
star subdivisions and a commutative diagram of stacks and toric
varieties
$$\begin{array}{ccc}
{\mathcal X}({\mathbf \Sigma'}) & \stackrel{f}\to & {\mathcal X}({\mathbf \Sigma})\\
\pi'\downarrow & & \pi \downarrow\\
X(\Sigma') & \stackrel{g}\to & X(\Sigma)
\end{array}
$$
where $f$ is birational, $\pi'$ is finite and $g$ is proper and birational.
\begin{definition}
  If $\alpha \in A^*({\mathcal X}({\mathbf \Sigma}))$ define $\pi_*
  \alpha := g_* \pi'_* f^* \alpha$. In particular if the associated
  toric variety $X(\Sigma)$ is complete we define for any $\alpha \in
  A^d({\mathcal X}({\mathbf \Sigma}))$ 
$$\int_{{\mathcal X}({\mathbf
      \Sigma})} \alpha := \int_{{\mathcal X}({\mathbf \Sigma'})}f^* \alpha$$
\end{definition}

\subsection{The Euler characteristic of an Artin toric stack with complete good moduli space}
We now turn to the definition of the Euler characteristic of an Artin
toric stack. If ${\mathcal X} = [X/G]$ is a smooth Artin stack then
the tangent bundle stack to ${\mathcal X}$ equals the vector bundle
stack ${\mathbb T}{\mathcal X}=[TX/\Lie(G)]$. The restriction of
${\mathbb T}{\mathcal X}$ to the open set ${\mathcal X}^{DM}$ where
${\mathcal X}$ is Deligne-Mumford is the usual tangent bundle to
${\mathcal X}^{DM}$.  Identifying the Chow groups of ${\mathcal X}$
with the $G$-equivariant Chow groups of $X$ we may define the Chern
series of ${\mathbb T}{\mathcal X}$ as the formal series $c_t({\mathbb
  T}{\mathcal X}) = c_t(TX)c_t(\Lie(G))^{-1} \in \prod_{i=0}^\infty
A^i_G(X) \otimes \ZZ[[t]]$. Moreover, when $G$ is diagonalizable the
Lie algebra of $G$ is a trivial $G$-module so $c_t(\Lie(G)) = 1$ so
the Chern series of ${\mathbb T}{\mathcal X}$ is actually a polynomial.

We can extend the definition of the Euler characteristic to toric Artin stacks.
\begin{definition}
Let ${\mathbf \Sigma}$ be a stacky fan and suppose that associated toric variety
$X(\Sigma)$ has dimension $d$ then we set
$\chi({\mathcal X}({\bf \Sigma})) := \int_{{\mathcal X}({\mathbf \Sigma})}c_d({\mathbb T}{\mathcal X})$.
\end{definition}

\subsection{Formulas for Euler characteristics of 3-dimensional toric Artin stacks}

As an application of Theorem
\ref{prop:general-equiv-self-intersection}, we derive a formula for
the Euler characteristic of a 3-dimensional toric Artin stack.

\begin{lemma}
Let $\mathbf{\Sigma}=(\ZZ^3, \Sigma, \{v_1, \ldots, v_n\})$ be a $3$-dimensional complete stacky fan with only one nonsimplicial cone, call it $\s$, and label the $v_i$ so that $\s=\RR_{\geq 0}\langle v_1, \ldots, v_s\rangle$. Let $\mathbf{\Sigma_\s}=(\ZZ^3, \Sigma_\s, \{v_0, v_1, \ldots, v_n\})$ be the stacky fan formed by stacky star subdivision of $\mathbf{\Sigma}$ with respect to $\s$, and let $f:\ts{\Sigma_\s} \to \ts{\Sigma}$ be the induced morphism. For $0 \leq i\leq n$, let $y_i \in A^*(\ts{\Sigma_\s})$ be the equivariant fundamental class of the coordinate hyperplanes restricted to Cox space $C(\Sigma_\s)$. Let $\Sigma_\s(3) \smallsetminus \Sigma(3)$ be the set of maximal cones in $\Sigma_\s$ that are not in $\Sigma$. Then the difference between the Euler characteristic of $\ts{\Sigma}$ and the Euler characteristic of its simplicialization $\ts{\Sigma_\s}$ is given by 
\begin{align} \label{eqn:difference-step1}
&\chi(\ts{\Sigma})-\chi(\ts{\Sigma_\s}) = \\
&\sum_{\tau \in \Sigma_\s(3) \smallsetminus \Sigma(3)}\frac{s-3}{D_{\tau, \Sigma_\s}} +   {s-1 \choose 2}\int_{\ts{\Sigma_\s}}y_0^2(\sum_{1\leq i \leq s} y_i) + {s \choose 3}\int_{\ts{\Sigma_\s}}y_0^3.
\end{align}
\end{lemma}

\begin{proof}
  With the notation as in \ref{subsec:Cox-constr}, for $1 \leq i \leq
  s$, we have $f^*(x_i)=y0+y_i$. Also, for $i \geq s$,
  $f^*(x_i)=y_i$. Then $\chi(\mc{X}(\mbf{\Sigma})) =
  \int_{\ts{\Sigma_\s}} f^*(\prod_{i=1}^n (1+x_i)) =
  \int_{\ts{\Sigma_\s}} (\prod_{i=1}^s (1+y_0+y_i))\prod_{i>s}
  (1+y_i)$

  By Proposition \ref{prop:chow-ring} and Equations
  \ref{eqn:stanley-reisner-ring}, \ref{eqn:stanley-reisner-ideal}, we
  have that $y_{i_1}\cdots y_{i_t} =0$ in $A^*_G(C)$ if $v_{i_1},
  \ldots, v_{i_t}$ are not contained in a cone. In particular,
  $y_0y_i=0$ for any $i>s$. Hence when we expand $(\prod_{i=1}^s
  (1+y_0+y_i))\prod_{i>s} (1+y_i)$ and collect the degree $3$, we get
  (note that some monomials in the following equation are zero):
$$\int_{\ts{\Sigma_\s}} \left( \sum_{0<i <j< k}  y_iy_jy_k + \sum_{0<j< k\leq s}  {s-2 \choose 1}y_0y_jy_k + \sum_{0< k \leq s}  {s-1 \choose 2}y_0^2y_k + {s \choose 3}y_0^3 \right)$$
We have $(\sum_{0<i <j< k}  y_iy_jy_k) + (\sum_{0<j< k\leq s}  y_0y_jy_k) = \chi(\ts{\Sigma_\s})$.  
\end{proof}

\begin{remark}\label{rem:more-than-one-cone}
The above proof easily extends to the case when $\Sigma$ has more than one nonsimplicial cone. Suppose  $\s_1$ and $\s_2$ are two different nonsimplicial cones of $\Sigma$, then the above proof shows that the equivariant fundamental class of the exceptional divisor of the blowup of $V(\s_1)$ (resp. $V(\s_2)$) has nonzero product only with divisors coming from rays of $\s_1$ (resp. $\s_2$).   
\end{remark}

By applying Theorem \ref{prop:general-equiv-self-intersection}, we can give a combinatorial formula to the self-intersection integrals on the right hand side of equation \ref{eqn:difference-step1}.

Fix $i$ satisfying $1 \leq i \leq s$. We will give a formula for $\int_{\ts{\Sigma_\s}} y_0^2y_i$.
Let $v_i^+$ and $v_i^-$ be the two lattice vectors (among $v_1, \ldots, v_s$) such that $\tau_i^+=\RR_{\geq 0}\langle v_i^+, v_0, v_i \rangle$ and $\tau_i^-=\RR_{\geq 0}\langle v_i^-, v_0, v_i \rangle$ are the two maximal cones in $\Sigma_\s$ having $\gamma_i=\RR_{\geq 0}\langle  v_0, v_i \rangle$ as a common face.
Since $\tau_i^+$ and $\tau_i^-$ are simplicial, there is a relation
\begin{equation}
\beta_i^+v_i^+ + b_{0,i}v_0 + b_iv_i + \beta_i^-v_i^- =0
\end{equation}
where $\beta_i^+, b_{0,i}, b_i, \beta_i^-$ are rational numbers. Then by Cor. \ref{cor:cox-result} 

\begin{equation}\label{eqn:double-self-intersection}
\int_{\ts{\Sigma_\s}} y_0^2y_i=\frac{b_{0,i}}{\beta_i^+D_{\tau_i^+, \mbf{\Sigma_\s}}}=\frac{b_{0,i}}{\beta_i^-D_{\tau_i^-, \mbf{\Sigma_\s}}}.
\end{equation}

Finally, to compute the triple self-intersection integral
$\int_{\ts{\Sigma_\s}}y_0^3$, one could use
Theorem \ref{prop:general-equiv-self-intersection}. However, because we
have restricted the dimension of $\Sigma$ to $3$, we can instead use a
relation of rational equivalence, and reduce the computation to the
case of double self intersection integrals just computed. Namely, pick
some $m_\s \in M$ such that $\langle m_\s, v_0 \rangle \neq 0$ (there
exists such an $m_\s$ since $v_0 \neq 0$ since $v_0$ lies in the
interior of $\s$). Then the rational equivalence relation
$div(\chi^{m_\s})=0$ can be written as $\sum_{0 \leq i \leq n} \langle
m_\s, v_i \rangle y_i = 0$. Hence
$$y_0=\frac{-1}{\langle m_\s, v_0 \rangle} \sum_{1 \leq i \leq n} \langle m_\s, v_i \rangle y_i.$$ 
Multiplying by $y_0^2$ and using the fact that $y_0y_i =0 $ for $i>s$ gives

$$\int_{\ts{\Sigma_\s}}y_0^3=\frac{-1}{\langle m_\s, v_0 \rangle} \int_{\ts{\Sigma_\s}}\sum_{1 \leq i \leq s} \langle m_\s, v_i \rangle y_0^2y_i$$

Using Remark \ref{rem:more-than-one-cone}, we combine the above formulas to deduce the following result:

\begin{theorem}\label{prop:final-3d-euler-char}
Let $\mathbf{\Sigma}=(\ZZ^3, \Sigma, \{ v_\rho \}_{\rho \in \Sigma(1)})$ be a $3$-dimensional complete stacky fan. Let ${NS}$ denote the set of nonsimplicial cones of $\Sigma$. Let $\mathbf{\Sigma_{simp}}=(\ZZ^3, \Sigma_{simp}, \{ v_\rho \}_{\rho \in \Sigma_{simp}(1)})$ be the stacky fan formed by stacky star subdivision of $\mathbf{\Sigma}$ with respect to the cones in ${NS}$, and let $f:\ts{\Sigma_{simp}} \to \ts{\Sigma}$ be the induced morphism. 

For $\s \in NS$, let $s_\s = \vert \s(1) \vert$ be the number of rays in $\s$, let $v_\s=\sum_{\rho \in \s(1)} v_\rho$, let $m_\s \in M$ be such that $\langle m_\s, v_\s \rangle \neq 0$, and regard $\Sigma_\s \subseteq \Sigma_{simp}$. For each $\rho \in \s(1)$ let $v_\rho^+$ and $v_\rho^-$ be the two lattice vectors such that $\tau_\rho^+=\RR_{\geq 0}\langle v_\rho^+, v_\s, v_\rho \rangle$ and $\tau_\rho^-=\RR_{\geq 0}\langle v_\rho^-, v_\s, v_\rho \rangle$ are the two maximal cones in $\Sigma_\s$ having $\gamma_\rho=\RR_{\geq 0}\langle  v_\s, v_\rho \rangle$ as a common face. Let $\beta_\rho^+, b_{\s,\rho}, b_\rho, \beta_\rho^-$ be rational numbers such that
\begin{equation}\label{eqn:wall-crossing}
\beta_\rho^+v_\rho^+ + b_{\s,\rho}v_\s + b_\rho v_\rho + \beta_\rho^-v_\rho^- =0
\end{equation}

Let $\Sigma_{\s}(3) \smallsetminus \Sigma(3)$ be the set of maximal cones in $\Sigma_{\s}$ that are not in $\Sigma$. Let $D_{\s, \mbf{\Sigma}}$ be as in Definition \ref{def:stacky-multiplicity}. Then the difference between the Euler characteristic of $\ts{\Sigma}$ and the Euler characteristic of its simplicialization $\ts{\Sigma_{simp}}$ is given by 

\begin{align*} 
&\chi(\ts{\Sigma})-\chi(\ts{\Sigma_{simp}}) = \\
&\sum_{\s \in NS} 
\left( \sum_{\tau \in \Sigma_{\s}(3) \smallsetminus \Sigma(3)}\frac{s_\s-3}{D_{\tau, \Sigma_\s}} +  \sum_{\rho \in \s(1)} \left({s_\s-1 \choose 2} - {s_\s \choose 3}\frac{\langle m_\s, v_\rho 
\rangle}{\langle m_\s, v_\s \rangle}\right) \frac{b_{\s,\rho}}{\beta_\rho^+D_{\tau_\rho^+, \mbf{\Sigma_\s}}} \right)
\end{align*}
\end{theorem}

As an example, we use this formula to compute the Euler characteristic
the nonsimplicial toric stack considered in Example
\ref{example:triple-intersection}.

\begin{example}\label{example:euler-char-changes}
  We continue with Example \ref{example:triple-intersection}, and let
  $f:\ts{\Sigma_\s} \to \ts{\Sigma}$ be the morphism induced by stacky
  star subdivision of $\Sigma$ with respect to $\s$. The
  multiplicities of the maximal cones are $D125=D235=2$,
  $D345=D145=1$, $D012=D023=7$, $D034=D014=5$. Hence by
  Prop. \ref{prop:simplicial-euler-char}, we
  have $$\chi(\ts{\Sigma_\s})=\frac{1}{7} + \frac{1}{7} + \frac{1}{5}
  + \frac{1}{5}+ \frac{1}{2} + \frac{1}{2} + \frac{1}{1} +
  \frac{1}{1}=\frac{129}{35}.$$ Also, $NS=\{ \s \}$ and $s_\s=4$,
  so $$\sum_{\tau \in \Sigma_{\s}(3) \smallsetminus
    \Sigma(3)}\frac{s_\s-3}{D_{\tau, \Sigma_\s}} = \frac{1}{7} +
  \frac{1}{7} + \frac{1}{5} + \frac{1}{5} = \frac{24}{35}.$$ Since
  $v_\s=v_0=(0,1,4)$, we may take $m_\s=(0,1,0)$ so $\langle m_\s,
  v_\s \rangle = 1$. There are four rays in $\s(1)$, namely $\rho_i =
  \RR_{\geq 0}\langle v_i \rangle$ for $i=1, 2, 3, 4$, and for each of
  them we compute

\begin{equation}\label{eqn:third-term}
\left({s_\s-1 \choose 2} - {s_\s \choose 3}\frac{\langle m_\s, v_\rho 
\rangle}{\langle m_\s, v_\s \rangle}\right) \frac{b_{\s,\rho}}{\beta_\rho^+D_{\tau_\rho^+, \mbf{\Sigma_\s}}}. 
\end{equation}

For $\rho=\rho_1$, we have $\gamma_{\rho_1}=\RR_{\geq 0}\langle v_0, v_1\rangle$, and take $v_\rho^+=v_2$, $v_\rho^-=v_4$, and set $\beta_\rho^+=1$. Then equation \ref{eqn:wall-crossing} becomes 
$v_2 + b_{\s,\rho_1}v_0 + b_{\rho_1} v_1 + \beta_4v_4 =0$. Substituting in the coordinates of each $v_i$, this equation becomes a system of three equations in the three unknowns $b_{\s,\rho_1}, b_{\rho_1}, \beta_4$. Solving we get $b_{\s, \rho_1}=\frac{-3}{5}$ (and $b_{\rho_1}=0, \beta_4=\frac{7}{5}$, but we won't need them). Hence for $\rho=\rho_1$, equation \ref{eqn:third-term} becomes

\begin{equation}
\left({4-1 \choose 2} - {4 \choose 3}\frac{0}{1}\right) \frac{\frac{-3}{5}}{1 \cdot 7} = \frac{-9}{35}. 
\end{equation}

Similarly computations for $\rho=\rho_2, \rho_3, \rho_4$ yield $\frac{20}{49}, \frac{-9}{35}, \frac{-14}{25}$ respectively.

Hence the Euler characteristic $\chi(\ts{\Sigma})=\frac{129}{35} +
  \frac{24}{35} - \frac{9}{35} + \frac{20}{49} - \frac{9}{35} -
  \frac{14}{25} = \frac{4539}{1225}$.

\end{example}

\end{document}